\documentclass[12pt]{article}
\usepackage[top=1in,bottom=1in,left=1in,right=1in]{geometry}
\usepackage{indentfirst}
\usepackage{amsfonts,amsmath,amsthm,amssymb}
\usepackage{cite}
\usepackage[hidelinks]{hyperref}
\usepackage{microtype}

\newtheorem{theorem}{Theorem}[section]
\newtheorem{lemma}[theorem]{Lemma}
\newtheorem{proposition}[theorem]{Proposition}
\newtheorem{corollary}[theorem]{Corollary}
\theoremstyle{remark}
\newtheorem{remark}[theorem]{Remark}
\newtheorem{example}[theorem]{Example}

\newcommand{\Z}{\mathbb Z}
\newcommand{\Q}{\mathbb Q}
\newcommand{\lc}{\operatorname{lc}}

\begin{document}
\title{Reducibility of
\texorpdfstring{$rx^m+p^e f(x)$}{rx^m+p^e f(x)}
for Large Primes $p$}
\author{
Weilin Zhang$^1$\thanks{E-mail: weilin@gzhu.edu.cn.}\quad
Hongjian Li$^2$\thanks{Corresponding author. E-mail: lhj@gdufs.edu.cn.}\\
{\small\textit{$^1$School of Mathematics and Information Science,
Guangzhou University,}}\\
{\small\textit{Guangzhou 510006, Guangdong, P. R. China}}\\
{\small\textit{$^2$School of Mathematics and Statistics,
Guangdong University of Foreign Studies,}}\\
{\small\textit{Guangzhou 510006, Guangdong, P. R. China}}
}

\date{}
 \maketitle

\noindent\textbf{Abstract}\quad
We study the reducibility over $\Q$ of
$F_{p,e}(x)=rx^m+p^ef(x)$, where $r\in\Z\setminus\{0\}$,
$f\in\Z[x]$, and $0\le m<n:=\deg f$, for primes $p$ above explicit
coefficient-dependent thresholds.  For arbitrary $e\ge1$, we determine
the degrees, endpoint $p$-adic valuations, reductions modulo $p$, and
heights of all nonconstant proper integral factors.  Writing
$\delta=\gcd(e,m,n)$, we show that every such factor has degree
$tn/\delta$ for some $1\le t\le\delta-1$; in particular, $\delta=1$
implies irreducibility.  For $e=2$ and $e=3$, we obtain effective
necessary and sufficient criteria in terms of at most two associated
binary quadratic forms and at most one associated binary cubic form,
respectively.  The quadratic criterion reduces to at most two fixed
Pell-type equations with prescribed second coordinate $p$.  In the
cubic case, Thue's theorem yields finiteness of the reducible primes
whenever the associated form is irreducible over $\Q$.

\medskip


\section{Introduction}\label{sec:intro}

A classical problem in number theory is to understand the irreducibility
of one-parameter linear combinations of fixed coprime polynomials.  Let
$P,Q\in\Z[x]$ be nonzero and coprime, with at least one of them
nonconstant.  The polynomial
$P(x)+TQ(x)$ is irreducible over $\Q(T)$, so Hilbert's irreducibility theorem yields
infinitely many integers $t$ for which $P(x)+t Q(x)$ is irreducible over $\Q$; see, for example,
\cite{Hilbert1892,Schinzel2000}.  This qualitative statement naturally
leads to a sharper arithmetic question: for prescribed classes of
specializations, can one determine when irreducibility holds for all
sufficiently large parameters and give an explicit threshold?

A fundamental result of Cavachi~\cite{Cavachi2000} states that if
$P,Q\in\Z[x]$ are coprime and $\deg P<\deg Q$, then
$P(x)+pQ(x)$ is irreducible over $\Q$ for all but finitely many primes $p$.
Explicit lower bounds for $p$ were subsequently obtained by Cavachi,
V\^aj\^aitu, and Zaharescu~\cite{CVZ2002}.  Polynomials of the form
$P(x)+p^eQ(x)$ were studied further by Bonciocat, Bugeaud, Cipu, and Mignotte
\cite{BBCM2013} and by Bonciocat~\cite{Bonciocat2016}.  These works
established large-prime irreducibility criteria with explicit
coefficient-dependent bounds, under suitable coprimality conditions
involving the exponent and the relevant degree data.  They provide strong
sufficient conditions for irreducibility for all sufficiently large primes,
but the stated criteria do not in general determine the possible
factorizations when the corresponding coprimality conditions fail.

This leads naturally to the complementary problem of describing
reducibility beyond these coprimality conditions.  We now specialize $Q$
to the monomial $rx^m$ while keeping $f$ arbitrary, and consider
\begin{equation}\label{eq:F-general}
F_{p,e}(x)=rx^m+p^ef(x),
\qquad
r\in\Z\setminus\{0\},\quad
f\in\Z[x],\quad
0\le m<n:=\deg f.
\end{equation}
When $m>0$, we assume $f(0)\ne0$, so that $x^m$ and $f(x)$ are coprime.
The family \eqref{eq:F-general} extends the constant-term case $m=0$ by
allowing the first summand to be supported at an arbitrary single degree,
while leaving $f$ unrestricted.  For sufficiently large primes $p$, this
one-term structure forces the $p$-adic Newton polygon of $F_{p,e}$ to have
one edge when $m=0$ and two nonhorizontal edges when $m>0$.  Thus the
case $m>0$ produces a two-edge Newton-polygon structure while retaining
enough rigidity for an explicit analysis of the possible factorizations.

The constant-term case $m=0$ provides the closest precedent.  For $e=2$,
Zhang, Yuan, and Zhou~\cite{ZYZ2024} obtained a necessary and sufficient
large-prime reducibility criterion involving a quadratic representation
and an associated discriminant condition.  A preliminary treatment of
the constant-term cubic case appeared in the first author's doctoral
dissertation~\cite{ZhangThesis2025}.  The present paper treats the full
monomial family $0\le m<n$ and derives the quadratic and cubic criteria
from a single effective factor-structure theorem valid for every
$e\ge1$.  For $m=0$ and $e=2$, our criterion recovers the published
constant-term result, while the cubic case is developed here in a fully
effective form.

Our main structural theorem shows that, for sufficiently large primes
$p$, every nonconstant proper integral factor of $F_{p,e}$ is governed by
$\delta:=\gcd(e,m,n)$ and an integer $t$ with
$1\le t\le\delta-1$.  In particular, its degree is $tn/\delta$, and the
theorem also controls its endpoint $p$-adic valuations, reduction modulo
$p$, and height.  Thus $\delta=1$ implies irreducibility.  For $e=2$ and
$e=3$, this structure yields effective necessary and sufficient
reducibility criteria above explicit thresholds.  They involve at most
two associated binary quadratic forms in the quadratic case and at most
one associated binary cubic form in the cubic case.  The former lead to
at most two fixed Pell-type equations with prescribed second coordinate
$p$; in the latter case, Thue's theorem gives finiteness of the reducible
primes whenever the associated form is irreducible.

The necessary-and-sufficient binary-form criteria are developed here only
for $e=2$ and $e=3$; the general factor-structure theorem remains valid
for arbitrary exponents.  Section~\ref{sec:factor-structure} proves this
theorem and its irreducibility consequence.  Sections~\ref{sec:e2} and
\ref{sec:e3} treat the quadratic and cubic cases, respectively, and
Section~\ref{sec:conclusion} discusses the obstacles to extending these
criteria to higher exponents.

\section{The structure of proper factors of \texorpdfstring{$F_{p,e}$}{Fpe}} \label{sec:factor-structure}

We begin by fixing notation used throughout the paper.  For a nonzero
polynomial $P(x)=\sum_{j=0}^{d}c_jx^j\in\mathbb C[x]$, define its
coefficient height by $H(P):=\max_{0\le j\le d}|c_j|$.
We write $\lc(P)$ for the leading coefficient of $P$, $[x^j]P$ for
the coefficient of $x^j$ in $P$, and $v_p(u)$ for the $p$-adic
valuation of $u\in\mathbb Q^\times$.

For the fixed data $r$ and $f$ in \eqref{eq:F-general}, set
$R:=1+|r|+H(f)$.  This quantity is independent of $p$ and $e$ and will serve as a uniform
Archimedean bound for the roots of $F_{p,e}$.  Since
$\deg F_{p,e}=n$, the case $n=1$ is linear and hence immediate.
We therefore assume $n\ge2$ throughout the remainder of the paper.

We begin with two standard algebraic and $p$-adic tools.  The first
allows us to replace a factorization over $\Q$ by one over $\Z$ without
changing the factor degrees.

\begin{lemma}[Integral factorization]\label{lem:integral}
Let $P\in\Z[x]$.  Suppose that $P=UV$ with nonconstant $U,V\in\Q[x]$.
Then there exist nonconstant $G,K\in\Z[x]$ such that
\[
P=GK,
\qquad
\deg G=\deg U,
\qquad
\deg K=\deg V.
\]
\end{lemma}

\begin{proof}
Write $U=\alpha G_0$ and $V=\beta K_0$, where
$\alpha,\beta\in\Q^\times$ and $G_0,K_0\in\Z[x]$ are primitive.
By Gauss's lemma, $G_0K_0$ is primitive and
$P=(\alpha\beta)G_0K_0$.
Write $\alpha\beta=a/b$ in lowest terms, with $a\in\Z$ and
$b\in\Z_{>0}$.  Since $P\in\Z[x]$, the integer $b$ divides every
coefficient of $aG_0K_0$.  As $\gcd(a,b)=1$, it follows that $b$
divides every coefficient of $G_0K_0$.  Since $G_0K_0$ is primitive,
$b=1$.  Thus $\alpha\beta\in\Z$.  Taking $G=(\alpha\beta)G_0$ and
$K=K_0$ gives the desired factorization, and multiplication by a nonzero scalar
does not change degree.
\end{proof}

We next fix the Newton-polygon terminology used below.  For
$P(x)=\sum_{j=0}^d a_jx^j\in\Q[x]$, the lower $p$-adic Newton polygon of
$P$ is the lower convex hull of the points $(j,v_p(a_j))$ with $a_j\ne0$.
Equivalently, one may set $v_p(0)=+\infty$.

If an edge has endpoints $(i,a)$ and $(j,b)$ with $i<j$, write
$(j-i,b-a)=d(u,v)$, where $d=\gcd(j-i,|b-a|)$ and
$\gcd(u,|v|)=1$.
We call $(u,v)$ the primitive edge vector of the edge and regard the
edge as consisting of $d$ copies of this vector.  With this notation,
Dumas's product theorem takes the following form.

\begin{lemma}[Dumas's product theorem]\label{lem:dumas}
Let $p$ be a prime, and let $P,Q\in\Q[x]$ satisfy $P(0)Q(0)\ne0$.
Then the multiset of primitive edge vectors of the lower $p$-adic
Newton polygon of $PQ$ is the disjoint union of the corresponding
multisets for $P$ and $Q$.
\end{lemma}

\begin{proof}
This is Dumas's classical product theorem for $p$-adic Newton polygons
\cite{Dumas1906}; see also \cite[Chapter~6]{Prasolov2004} and the
formulation in \cite[Section~1]{BBCM2013}.  The stated multiset
formulation follows by decomposing each edge into copies of its primitive
edge vector.  Vertical translations of the polygons do not affect these
vectors.
\end{proof}

The next estimate provides the Archimedean input for the factor analysis.

\begin{lemma}[Uniform root bounds]\label{lem:annulus}
Let $F_{p,e}$ be as in \eqref{eq:F-general}.  Then every complex root
$\theta$ of $F_{p,e}$ satisfies $|\theta|\le R$. If $m>0$, then in addition
\[
R^{-1}\le |\theta|\le R.
\]
\end{lemma}

\begin{proof}
Write $f(x)=\sum_{j=0}^{n}a_jx^j$, where $a_n\ne0$.
We use the standard Cauchy bound: if
$P(x)=c_dx^d+c_{d-1}x^{d-1}+\cdots+c_0$ with $c_d\ne0$, then every
complex root $\zeta$ of $P$ satisfies
\begin{equation}\label{eq:cauchy-bound}
|\zeta|
\le
1+\max_{0\le j<d}\left|\frac{c_j}{c_d}\right|.
\end{equation}
Since $m<n$, we have $\lc(F_{p,e})=p^ea_n$.  For $j<n$ with
$j\ne m$, one has
$|[x^j]F_{p,e}/(p^ea_n)|=|a_j/a_n|\le H(f)$, while at degree $m$,
$|(r+p^ea_m)/(p^ea_n)|\le |r|+H(f)$.  Thus
\eqref{eq:cauchy-bound} gives $|\theta|\le1+|r|+H(f)=R$.

Assume now that $m>0$.  Then $a_0=f(0)\ne0$, so every root of
$F_{p,e}$ is nonzero.  The reciprocal polynomial
\[
F_{p,e}^*(x)
=
x^nF_{p,e}(x^{-1})
=
p^e\sum_{j=0}^n a_jx^{n-j}+rx^{n-m}
\]
has leading coefficient $p^ea_0$.  The same coefficient estimate gives
$|\eta|\le R$ for every root $\eta$ of $F_{p,e}^*$.  Since
$\theta^{-1}$ is a root of $F_{p,e}^*$, it follows that
$|\theta|\ge R^{-1}$.
\end{proof}

Define the common structural threshold
\begin{equation}\label{eq:Pstar}
P_*:=H(f)R^{\lfloor n/2\rfloor}.
\end{equation}
Since $n\ge2$ and $H(f)\ge1$, one has $P_*\ge R$.  Hence $p>P_*$
implies $p>R>\max\{|r|,H(f)\}$.

The following theorem gives the structural information on proper
integral factors used throughout the paper.

\begin{theorem}[Structure of proper factors]\label{thm:factor-structure}
Let $F_{p,e}$ be as in \eqref{eq:F-general}, where $p>P_*$ is prime,
and put $\delta:=\gcd(e,m,n)$.  Let $G$ be a divisor of $F_{p,e}$ in
$\Z[x]$ satisfying $0<\deg G<n$.  Then there exists an integer $t$
with $1\le t\le\delta-1$ such that:

\begin{enumerate}
\item[\textup{(i)}]
$\displaystyle \deg G=\frac{tn}{\delta}$;

\item[\textup{(ii)}]
$\displaystyle
v_p(\lc(G))=\frac{te}{\delta},
\qquad
v_p(G(0))=
\begin{cases}
0, & m=0,\\[2mm]
\dfrac{te}{\delta}, & 0<m<n
\end{cases}$;

\item[\textup{(iii)}]
$\displaystyle
G(x)\equiv bx^{tm/\delta}\pmod p$
for some $b\in\Z$ with $p\nmid b$;

\item[\textup{(iv)}]
$\displaystyle
H(G)\le
p^{te/\delta}H(f)(1+R)^{tn/\delta}$.
\end{enumerate}
\end{theorem}

\begin{proof}
Write $F_{p,e}=GK$ with $K\in\Z[x]$, and put $d:=\deg G$.

\smallskip
\noindent
\emph{Case 1: $m=0$.}
Since $p>R>\max\{|r|,H(f)\}$, we have
$v_p(F_{p,e}(0))=v_p(r+p^ef(0))=0$, whereas every nonzero coefficient
of positive degree has $p$-adic valuation $e$.  Hence the lower
$p$-adic Newton polygon of $F_{p,e}$ is the single edge
$(0,0)\longrightarrow(n,e)$.  Here $\delta=\gcd(e,n)$, so this edge
consists of $\delta$ copies of the primitive vector
$\left(n/\delta,e/\delta\right)$.
By Lemma~\ref{lem:dumas}, the Newton polygon of $G$ contains exactly $t$
copies of this vector for some integer $t$.  Since both $G$ and $K$ are
nonconstant, $1\le t\le\delta-1$.
Its horizontal width and total vertical rise therefore give
$d=tn/\delta$ and
$v_p(\lc(G))-v_p(G(0))=te/\delta$.
Since $v_p(F_{p,e}(0))=0$ and
$F_{p,e}(0)=G(0)K(0)$, both endpoint valuations at the left are
nonnegative and hence $v_p(G(0))=0$.
This proves parts \textup{(i)} and \textup{(ii)} in the present case.
The unique lowest point of the Newton polygon of
$G$ is its left endpoint, so every positive-degree coefficient of $G$ is
divisible by $p$: otherwise a point $(j,0)$ with $j>0$ would lie below
the positive-slope edge of the polygon.  Thus
$G(x)\equiv b\pmod p$ for some $b\in\Z$ with $p\nmid b$, which proves part \textup{(iii)}
because $m=0$.

\smallskip
\noindent
\emph{Case 2: $0<m<n$.}
We first prove the endpoint balance
\begin{equation}\label{eq:proper-factor-balance}
v_p(G(0))=v_p(\lc(G)).
\end{equation}
Assume initially that $d\le\lfloor n/2\rfloor$.  Put
$c:=v_p(G(0))$ and $s:=v_p(\lc(G))$, and write $G(0)=p^cu$ and
$\lc(G)=p^sv$, where $p\nmid uv$.  Since $p>H(f)$, we have
$v_p(f(0))=v_p(\lc(f))=0$.  From $G(0)K(0)=p^ef(0)$ and
$\lc(G)\lc(K)=p^e\lc(f)$, and since $p\nmid f(0)\lc(f)$, we may write
$K(0)=p^{e-c}u'$ and $\lc(K)=p^{e-s}v'$ with
$u',v'\in\Z$ not divisible by $p$.  Then $uu'=f(0)$ and
$vv'=\lc(f)$, so the $p$-free parts $u$ and $v$ divide $f(0)$ and
$\lc(f)$, respectively.  Hence $1\le|u|,|v|\le H(f)$.
Let $\theta_1,\ldots,\theta_d$ be the roots of $G$, counted with
multiplicity.  Lemma~\ref{lem:annulus} gives
\[
R^{-d}
\le
\left|\frac{G(0)}{\lc(G)}\right|
=
p^{c-s}\left|\frac uv\right|
\le
R^d.
\]
If $c>s$, then the middle term is at least
\[
\frac{p}{H(f)}
>
R^{\lfloor n/2\rfloor}
\ge R^d,
\]
a contradiction.  If $c<s$, it is at most
\[
\frac{H(f)}p
<
R^{-\lfloor n/2\rfloor}
\le R^{-d},
\]
again a contradiction.  Thus $c=s$.

If $d>n/2$, then $\deg K<n/2$, so the preceding argument applied to $K$
gives $v_p(K(0))=v_p(\lc(K))$.  Since
$v_p(F_{p,e}(0))=v_p(\lc(F_{p,e}))=e$, the same equality follows for
$G$.  This proves
\eqref{eq:proper-factor-balance} for every proper factor.

Now put $d_-:=\gcd(e,m)$ and $d_+:=\gcd(e,n-m)$.
Because $p>R>\max\{|r|,H(f)\}$, the lower $p$-adic Newton polygon of
$F_{p,e}$ has exactly the vertices $(0,e)$, $(m,0)$, and $(n,e)$.
Its primitive edge vectors consist of $d_-$ copies of
$\left(m/d_-,-e/d_-\right)$ and $d_+$ copies of
$\left((n-m)/d_+,e/d_+\right)$.
By Lemma~\ref{lem:dumas}, suppose that the Newton polygon of $G$ contains
$\mu$ copies of the first vector and $\nu$ copies of the second.  The
balance relation \eqref{eq:proper-factor-balance} says that its total
vertical displacement is zero, so $\mu/d_-=\nu/d_+$.  Moreover,
$\gcd(d_-,d_+)=\gcd(e,m,n-m)=\gcd(e,m,n)=\delta$.
Write $d_-=\delta u_-$ and $d_+=\delta u_+$, where
$\gcd(u_-,u_+)=1$.  Then $\mu u_+=\nu u_-$, so
$\mu=tu_-$ and $\nu=tu_+$ for some integer $t\ge1$.  Consequently,
$d=\mu m/d_-+\nu(n-m)/d_+=tn/\delta$.
Since $G$ is proper, $d<n$, and hence $1\le t\le\delta-1$.  This proves
part \textup{(i)}.

Since $F_{p,e}(x)\equiv rx^m\pmod p$ is nonzero, neither $G$ nor $K$
is divisible by $p$.  Thus the minimum
ordinate of the Newton polygon of $G$ is $0$.  Its negative-slope part
has total vertical drop $\mu e/d_-=te/\delta$, and the positive-slope
part has the same total rise.  Together with
\eqref{eq:proper-factor-balance}, this gives
$v_p(G(0))=v_p(\lc(G))=te/\delta$, proving part \textup{(ii)}.  The unique lowest vertex occurs at
abscissa $\mu m/d_-=tm/\delta$.
The coefficient corresponding to this vertex is a $p$-adic unit.  If a
coefficient at any other degree were also a $p$-adic unit, the associated
point of ordinate $0$ would lie below one of the two nonhorizontal edges
of the Newton polygon of $G$.  Hence precisely the coefficient at degree
$tm/\delta$ is a $p$-adic unit, while every other coefficient is divisible
by $p$.  This proves part \textup{(iii)}.

It remains to prove the height bound, in either case.  By
part \textup{(ii)}, write $\lc(G)=p^{te/\delta}w$, where $p\nmid w$.
Since $G\mid F_{p,e}$ in $\Z[x]$ and $p>H(f)$, the integer $w$ divides
$\lc(f)$; hence $|w|\le|\lc(f)|\le H(f)$.
Let $\theta_1,\ldots,\theta_d$ be the roots of $G$.  By
Lemma~\ref{lem:annulus}, $|\theta_i|\le R$.  Therefore, for
$0\le k\le d$, Vi\`ete's formulas give
\[
\begin{aligned}
\bigl|[x^{d-k}]G\bigr|
&\le
p^{te/\delta}|w|\binom dk R^k\\
&\le
p^{te/\delta}H(f)(1+R)^d.
\end{aligned}
\]
Using $d=tn/\delta$ proves part \textup{(iv)}.
\end{proof}

The structural theorem immediately yields the basic irreducibility criterion
for the general exponent.

\begin{corollary}[Large-prime irreducibility criterion]\label{cor:delta-one}
Let $F_{p,e}$ be as in \eqref{eq:F-general}, and let $p>P_*$ be prime.
If $\gcd(e,m,n)=1$, then $F_{p,e}$ is irreducible over $\Q$.
\end{corollary}

\begin{proof}
Suppose that $F_{p,e}$ is reducible over $\Q$.  By
Lemma~\ref{lem:integral}, there is a nontrivial factorization
$F_{p,e}=GK$ with $G,K\in\Z[x]$.
Then $0<\deg G<n$, so Theorem~\ref{thm:factor-structure} yields an
integer $t$ satisfying $1\le t\le\delta-1$, where
$\delta=\gcd(e,m,n)$.
If $\delta=1$, this is impossible.
\end{proof}

Corollary~\ref{cor:delta-one} may be compared with earlier large-prime
irreducibility criteria for sums of coprime polynomials.  For the family
\eqref{eq:F-general}, the criterion of Bonciocat, Bugeaud, Cipu, and
Mignotte~\cite{BBCM2013} applies, in particular, under
the condition $\gcd(e,n-m)=1$, whereas the criterion of
Bonciocat~\cite{Bonciocat2016} applies under $\gcd(e,n)=1$.
Each of these conditions implies $\gcd(e,m,n)=1$.
Indeed, any common divisor of $e$, $m$, and $n$ also divides $n-m$.
The converse, however, need not hold.  The following example gives an
infinite family for which neither of the preceding coprimality conditions
holds, while Corollary~\ref{cor:delta-one} still applies.

\begin{example}\label{ex:delta-one-family}
Let $k\ge0$, let $r\in\Z\setminus\{0\}$, and let $f\in\Z[x]$ have
degree $6k+3$ with $f(0)\ne0$.  Consider
\[
F_{p,6}(x)=rx^{6k+1}+p^6f(x).
\]
Since $\gcd(6,6k+1,6k+3)=1$,
Corollary~\ref{cor:delta-one} shows that $F_{p,6}$ is irreducible over
$\Q$ whenever
\[
p>P_*
=
H(f)\bigl(1+|r|+H(f)\bigr)^{3k+1}.
\]
On the other hand, $\gcd(6,6k+3)=3$ and
$\gcd\bigl(6,(6k+3)-(6k+1)\bigr)=\gcd(6,2)=2$.
Thus neither of the two preceding coprimality conditions is satisfied.
\end{example}

\section{Effective reducibility criteria via binary quadratic forms for
\texorpdfstring{$F_{p,2}$}{Fp2}}
\label{sec:e2}

We now specialize to the case $e=2$, so that
$F_{p,2}(x)=rx^m+p^2f(x)$.
By Corollary~\ref{cor:delta-one}, if at least one of $m$ and $n$ is odd,
then $F_{p,2}$ is irreducible over $\Q$ for every prime $p>P_*$.  Hence
the only case requiring further analysis is $2\mid m$ and $2\mid n$.
Accordingly, throughout this section we write $m=2M$ and $n=2N$, and
put $W:=x^M$ and $a:=\lc(f)$.

The proof of the criterion below will show that reducibility forces $f$
to admit a quadratic representation in a monic degree-$N$ polynomial and
$W=x^M$.  We therefore introduce the corresponding binary quadratic
forms.

For $b,c\in\Z$, put $\Phi_{b,c}(X,Y):=X^2-bXY+acY^2$.
Define $\mathcal Q_2(f;m)$ to be the set of all such forms for which
there exists a monic polynomial $h\in\Q[x]$ satisfying $\deg h=N$ and
$[x^M]h=0$, and such that $\Phi_{b,c}(ah,-W)=af$.  Equivalently,
$f=ah^2+bWh+cW^2$.
We call the elements of $\mathcal Q_2(f;m)$ the quadratic forms
associated with $f$ and $m$, and say that such a polynomial $h$
\emph{realizes} the corresponding form.  By construction,
$\mathcal Q_2(f;m)$ depends only on $f$ and $m$, and is independent of
$r$ and $p$.

Put
\begin{equation}\label{eq:C2}
C_2:=H(f)(1+R)^N+1
\end{equation}
and define
\begin{equation}\label{eq:P2}
P_2(r,f)
:=
\max\{2,P_*,2H(f)C_2\}.
\end{equation}

\begin{theorem}[Quadratic-form criterion]
\label{thm:e2-main}
Let $p>P_2(r,f)$ be prime.  Then $F_{p,2}$ is reducible over $\Q$
if and only if there exists $\Phi\in\mathcal Q_2(f;m)$ such that
$\Phi(z,p)=-ar$ for some $z\in\Z$.
\end{theorem}

\begin{proof}
Suppose first that there exists $\Phi\in\mathcal Q_2(f;m)$ such that
$\Phi(z,p)=-ar$ for some $z\in\Z$.  Write
$\Phi(X,Y)=X^2-bXY+acY^2$ with $b,c\in\Z$, and choose a monic
polynomial $h\in\Q[x]$ realizing $\Phi$, so that $\deg h=N$,
$[x^M]h=0$, and $f=ah^2+bWh+cW^2$.  The relation
$\Phi(z,p)=-ar$ is equivalent to $z(bp-z)=acp^2+ar$.
Hence
\[
\begin{aligned}
aF_{p,2} 
& = p^2af+arx^m \\
& = p^2a\bigl(ah^2+bWh+cW^2\bigr)+arW^2 \\
& = p^2a^2h^2+p^2abWh+z(bp-z)W^2\\
& = (pah+zW)\bigl(pah+(bp-z)W\bigr).
\end{aligned}
\]
Since $a\ne0$ and $N>M$, the two factors on the right have degree $N$.
Thus $F_{p,2}$ is reducible over $\Q$.

Conversely, suppose that $F_{p,2}$ is reducible over $\Q$.

\smallskip
\noindent
\emph{Step 1: a half-degree factorization.}
Since $p>P_2(r,f)\ge P_*$, Lemma~\ref{lem:integral} gives a
nontrivial factorization
\begin{equation}\label{eq:q2-GK}
F_{p,2}=GK,
\qquad
G,K\in\Z[x].
\end{equation}
Here $\delta=\gcd(2,m,n)=2$.  Applying
Theorem~\ref{thm:factor-structure} to both proper factors forces $t=1$ and
gives $\deg G=\deg K=N$ and
$v_p(\lc(G))=v_p(\lc(K))=1$.  Moreover,
$G\equiv uW\pmod p$ and $K\equiv vW\pmod p$ for some integers $u,v$
with $p\nmid uv$, while $H(G),H(K)\le pH(f)(1+R)^N$.  Choose
representatives modulo $p$ such that $|u|,|v|<p/2$, and write
\begin{equation}\label{eq:q2-GK-decomposition}
G=pA+uW,
\qquad
K=pC+vW,
\end{equation}
with $A,C\in\Z[x]$ of degree $N$.

Put $\alpha:=\lc(A)$ and $\beta:=\lc(C)$.  Since
$\lc(G)=p\alpha$, $\lc(K)=p\beta$, and $\lc(F_{p,2})=p^2a$, we have
\begin{equation}\label{eq:q2-alphabeta}
\alpha\beta=a.
\end{equation}
In particular, $1\le|\alpha|,|\beta|\le H(f)$.

For $j\ne M$, \eqref{eq:q2-GK-decomposition} gives
$[x^j]A=[x^j]G/p$ and $[x^j]C=[x^j]K/p$, whereas at degree $M$,
$[x^M]A=([x^M]G-u)/p$ and $[x^M]C=([x^M]K-v)/p$.
Since $|u|,|v|<p/2$, it follows that
$H(A),H(C)<H(f)(1+R)^N+1/2$.
Hence
\begin{equation}\label{eq:q2-AC-height}
H(A)<C_2,
\qquad
H(C)<C_2.
\end{equation}

\smallskip
\noindent
\emph{Step 2: coefficient alignment.}
Expanding \eqref{eq:q2-GK-decomposition}, we obtain
$F_{p,2}=p^2AC+p(uC+vA)W+uvW^2$.
For $0\le i\le N$ with $i\ne M$, one has $M+i\ne2M=m$.
Hence the coefficient of $x^{M+i}$ in $F_{p,2}$ is divisible by $p^2$,
and comparison modulo $p^2$ gives
\begin{equation}\label{eq:q2-alignment-congruence}
u[x^i]C+v[x^i]A\equiv0\pmod p.
\end{equation}
Taking $i=N$ in \eqref{eq:q2-alignment-congruence} yields
$u\beta+v\alpha\equiv0\pmod p$.
Multiplying \eqref{eq:q2-alignment-congruence} by $\alpha$ and the last
congruence by $[x^i]A$, and then subtracting, gives
$u(\alpha[x^i]C-\beta[x^i]A)\equiv0\pmod p$.  Since $p\nmid u$,
$\alpha[x^i]C\equiv\beta[x^i]A\pmod p$ for $i\ne M$.
By \eqref{eq:q2-AC-height},
$|\alpha[x^i]C-\beta[x^i]A|<2H(f)C_2<p$, and hence
$\alpha[x^i]C=\beta[x^i]A$ for $i\ne M$.

Set $A_0:=A-[x^M]A\,W$ and $C_0:=C-[x^M]C\,W$.  Then
$H_0:=\beta A_0=\alpha C_0\in\Z[x]$.  Its leading coefficient is
$\alpha\beta=a$, so $h:=H_0/a$ is monic of degree $N$ and satisfies
$[x^M]h=0$.

\smallskip
\noindent
\emph{Step 3: reconstruction of the quadratic form.}
Put $U:=u+p[x^M]A$ and $V:=v+p[x^M]C$.  Then
$G=pA_0+UW$ and $K=pC_0+VW$.
Comparison of the coefficient of $x^{N+M}$ in $GK=F_{p,2}$ gives
$p\mid\beta U+\alpha V$, whereas comparison at degree $2M$, using
$[x^M]A_0=[x^M]C_0=0$, gives $UV\equiv r\pmod{p^2}$.  Hence
$b:=(\beta U+\alpha V)/p$ and $c:=(UV-r)/p^2$ are integers.

Since $H_0=ah$ and $a=\alpha\beta$,
\[
\begin{aligned}
aF_{p,2}
&=(pH_0+\beta UW)(pH_0+\alpha VW)\\
&=arW^2
  +p^2a\bigl(ah^2+bWh+cW^2\bigr).
\end{aligned}
\]
Comparing this with $aF_{p,2}=arW^2+p^2af$ gives
$f=ah^2+bWh+cW^2$.  Therefore
$\Phi(X,Y):=X^2-bXY+acY^2$ belongs to $\mathcal Q_2(f;m)$.

Finally, set $z:=\beta U\in\Z$.  Since
$bp=\beta U+\alpha V$, we have $bp-z=\alpha V$, and therefore
\[
\begin{aligned}
\Phi(z,p)
&=z^2-bpz+acp^2\\
&=\beta^2U^2-\beta U(\beta U+\alpha V)
  +a(UV-r)\\
&=-ar.
\end{aligned}
\]
Thus the required $\Phi\in\mathcal Q_2(f;m)$ and $z\in\Z$ exist.
\end{proof}

Theorem~\ref{thm:e2-main} reduces the reducibility problem for large primes
to the quadratic forms associated with $f$ and $m$.
In particular, if $\mathcal Q_2(f;m)=\varnothing$, then
$F_{p,2}$ is irreducible over $\Q$ for every prime
$p>P_2(r,f)$.  Thus the construction of the associated set is a finite
polynomial preprocessing step that may already settle the problem before
any Diophantine equation is considered.
The next proposition shows that at most two such forms can occur and
determines exactly when the upper bound is attained.

\begin{proposition}[Cardinality of the associated quadratic forms]
\label{prop:q2-cardinality}
The set $\mathcal Q_2(f;m)$ satisfies
$|\mathcal Q_2(f;m)|\le2$.
Equality holds if and only if
\begin{equation}\label{eq:q2-critical-f}
f(x)=ax^{2m}+\beta x^m+\gamma,
\qquad
a\gamma=\rho^2\ne0
\end{equation}
for some $\beta,\gamma,\rho\in\Z$.
In this case,
\begin{equation}\label{eq:q2-critical-forms}
\mathcal Q_2(f;m)
=
\left\{
X^2+a(\beta-2\rho)Y^2,\,
X^2+a(\beta+2\rho)Y^2
\right\}.
\end{equation}
\end{proposition}

\begin{proof}
We first prove the upper bound $|\mathcal Q_2(f;m)|\le2$.  Write
$f(x)=\sum_{j=0}^{2N}f_jx^j$, where $a=f_{2N}$, and suppose that
$\Phi(X,Y)=X^2-bXY+acY^2\in\mathcal Q_2(f;m)$ is realized by
$h(x)=x^N+\sum_{j=0}^{N-1}u_jx^j$ with $u_M=0$.
Equivalently,
\begin{equation}\label{eq:q2-representation-for-cardinality}
f=ah^2+bWh+cW^2.
\end{equation}

For $j=N-1,N-2,\ldots,M+1$, the terms $bWh$ and $cW^2$ do not contribute to the coefficient of
$x^{N+j}$.  Hence
$f_{N+j}/a=2u_j+\sum_{i=j+1}^{N-1}u_i u_{N+j-i}$, and therefore
\begin{equation}\label{eq:q2-upper-recovery}
u_j
=
\frac12
\left(
\frac{f_{N+j}}a
-
\sum_{i=j+1}^{N-1}u_i u_{N+j-i}
\right).
\end{equation}
Thus $u_{N-1},u_{N-2},\ldots,u_{M+1}$ are uniquely determined,
successively from the coefficients of $f$.

If $M=0$, then the normalization gives $u_0=0$.  Hence there is at most one candidate $h$ satisfying the defining conditions.

Assume now that $M>0$.  Since neither $bWh$ nor $cW^2$ contributes to
the constant term of \eqref{eq:q2-representation-for-cardinality}, one has
\begin{equation}\label{eq:q2-u0-recovery}
u_0^2=\frac{f_0}{a}.
\end{equation}
Thus there are at most two rational possibilities for $u_0$.  Moreover,
$f(0)\ne0$, so $u_0\ne0$.  Once $u_0$ is fixed, comparison of the
coefficients of $x^j$, for $1\le j\le M-1$, gives
$f_j/a=2u_0u_j+\sum_{i=1}^{j-1}u_i u_{j-i}$, and hence
\begin{equation}\label{eq:q2-lower-recovery}
u_j
=
\frac{
f_j/a-\sum_{i=1}^{j-1}u_i u_{j-i}
}{
2u_0
}.
\end{equation}
Therefore $u_1,\ldots,u_{M-1}$ are uniquely determined once $u_0$ is
chosen.  Together with
$u_M=0$ and \eqref{eq:q2-upper-recovery}, this determines $h$
uniquely for each of the at most two choices of $u_0$.

For a fixed $h$, the coefficients $b$ and $c$ are also uniquely
determined.  Indeed, comparison in
\eqref{eq:q2-representation-for-cardinality} at degrees $N+M$ and $2M$
gives
\begin{equation}\label{eq:q2-bc-recovery}
b
=
f_{N+M}-a[x^{N+M}]h^2,
\qquad
c
=
f_{2M}-a[x^{2M}]h^2,
\end{equation}
where the normalization $[x^M]h=0$ is used in the second identity.
It follows that each admissible $h$ determines at most one element of
$\mathcal Q_2(f;m)$.  Hence $|\mathcal Q_2(f;m)|\le2$.

We next determine when the upper bound is attained.  Suppose that
$\mathcal Q_2(f;m)$ contains two distinct forms.  Then there are
polynomials $h,h'\in\Q[x]$, each satisfying the defining conditions,
and coefficients $b,c,b',c'\in\Z$ such that
\[
f=ah^2+bWh+cW^2
=
ah'^2+b'Wh'+c'W^2.
\]
Put $H_1=2ah+bW$ and $D=b^2-4ac$, and similarly put
$H_2=2ah'+b'W$ and $D'=b'^2-4ac'$.  Then
$4af=H_1^2-DW^2=H_2^2-D'W^2$, and hence
\begin{equation}\label{eq:q2-keyproduct}
(H_1-H_2)(H_1+H_2)
=
(D-D')W^2.
\end{equation}

If $D=D'$, then \eqref{eq:q2-keyproduct} gives $H_1=H_2$, since
$H_1+H_2$ has leading coefficient $4a\ne0$.  Thus
$2a(h-h')+(b-b')W=0$.  Taking the coefficient of $x^M$ and using
$[x^M]h=[x^M]h'=0$
gives $b=b'$, and then $h=h'$ and $c=c'$.  This contradicts the
assumption that the two associated forms are distinct.  Therefore
$D\ne D'$.

The right-hand side of \eqref{eq:q2-keyproduct} is now a nonzero scalar
multiple of $x^{2M}$.  Since the product of two nonzero polynomials in $\Q[x]$ is a monomial,
each factor must itself be a monomial.  Since $H_1+H_2$ has degree $N$
and leading coefficient $4a$, one has $H_1+H_2=4ax^N$ and
$H_1-H_2=\lambda x^{2M-N}$ for some $\lambda\in\Q^\times$.  In
particular, $N\le2M$.
The case $M=0$ has already been shown to give at most one form, so
$M>0$.  If $N<2M$, then both $H_1+H_2$ and $H_1-H_2$ have zero constant
term, and hence
$H_1(0)=H_2(0)=0$.  Since $W(0)=0$, the identity
$4af=H_1^2-DW^2$ would imply $f(0)=0$, contrary to the standing
hypothesis.  Thus $N=2M$, which is equivalent to $n=2m$.

Now $H_1-H_2$ is constant.  Comparing the coefficients of $x^M$ in
$H_1+H_2=4ax^N$ and in $H_1-H_2$ gives $b+b'=b-b'=0$, and therefore
$b=b'=0$.  It follows that $h+h'=2x^N$, whereas $h-h'$ is a nonzero
constant.  Since $N=m$, there exists $t\in\Q^\times$ such that
$h=x^m+t$ and $h'=x^m-t$.  Substitution into either representation gives
$f(x)=ax^{2m}+\beta x^m+\gamma$ for some $\beta,\gamma\in\Z$, with
$\gamma=at^2$.  Putting $\rho:=at$, we obtain $\rho^2=a\gamma$.
Since $\rho\in\Q$ and $\rho^2\in\Z$, writing $\rho=r_0/s_0$ in lowest
terms shows that $s_0=1$; hence $\rho\in\Z$.
Moreover, $t\ne0$, so $\rho\ne0$.  Thus
\eqref{eq:q2-critical-f} holds.

Conversely, suppose that \eqref{eq:q2-critical-f} holds.  Then
\[
\begin{aligned}
f(x)
&=
a\left(x^m+\frac{\rho}{a}\right)^2
+(\beta-2\rho)x^m,\\
f(x)
&=
a\left(x^m-\frac{\rho}{a}\right)^2
+(\beta+2\rho)x^m.
\end{aligned}
\]
Since $n=2m$, one has $N=m=2M$, and both polynomials
$x^m\pm\rho/a$ are monic of degree $N$ with zero $x^M$ coefficient.
Hence both satisfy the defining conditions and yield the two forms
displayed in \eqref{eq:q2-critical-forms}.  They are distinct because $\rho\ne0$.
Since we have already proved that
$|\mathcal Q_2(f;m)|\le2$, there can be no further associated form.
This completes the proof.
\end{proof}

\medskip
\noindent\textbf{Effective construction of $\mathcal Q_2(f;m)$.}
The proof of Proposition~\ref{prop:q2-cardinality} is constructive.
It determines, from $f$ and $m$, at most two candidate polynomials $h$;
for each candidate, the coefficients $b$ and $c$ are uniquely determined,
and the candidate is retained precisely when $b,c\in\Z$ and
$f=ah^2+bWh+cW^2$.
Hence $\mathcal Q_2(f;m)$ is effectively computable by at most two final
identity checks.  In particular, this polynomial preprocessing depends only
on $f$ and $m$, and is independent of $p$ and $r$.

It remains to translate the arithmetic representation in
Theorem~\ref{thm:e2-main} into a classical Diophantine form.

For an associated quadratic form
$\Phi(X,Y)=X^2-bXY+acY^2\in\mathcal Q_2(f;m)$, we write
$\Delta(\Phi):=b^2-4ac$ for its discriminant, and associate with $\Phi$
the Pell-type equation
\begin{equation}\label{eq:q2-associated-pell}
S^2-\Delta(\Phi)T^2=-4ar.
\end{equation}

\begin{corollary}[Pell-type reducibility criterion]
\label{cor:q2-pell}
Let $p>P_2(r,f)$ be prime.  Then $F_{p,2}$ is reducible over $\Q$
if and only if there exists $\Phi\in\mathcal Q_2(f;m)$ such that the associated Pell-type equation
\eqref{eq:q2-associated-pell} has an integral solution $(S,T)=(s,p)$
for some $s\in\Z$.
\end{corollary}

\begin{proof}
Let $\Phi(X,Y)=X^2-bXY+acY^2\in\mathcal Q_2(f;m)$.  By the definition
of $\Delta(\Phi)$,
$4\Phi(X,Y)=(2X-bY)^2-\Delta(\Phi)Y^2$.

By Theorem~\ref{thm:e2-main}, $F_{p,2}$ is reducible over $\Q$
if and only if there exist $\Phi\in\mathcal Q_2(f;m)$ and $z\in\Z$
such that $\Phi(z,p)=-ar$.  For such a pair, setting $s=2z-bp$ gives
$s^2-\Delta(\Phi)p^2=-4ar$,
so \eqref{eq:q2-associated-pell} has the integral solution $(s,p)$.

Conversely, suppose that \eqref{eq:q2-associated-pell} has an integral
solution $(s,p)$ for some $\Phi\in\mathcal Q_2(f;m)$.  Since
$\Delta(\Phi)\equiv b^2\pmod4$ and $p>P_2(r,f)\ge2$ is odd, reduction
modulo $4$ gives $s\equiv bp\pmod2$.  Hence
$z:=(s+bp)/2\in\Z$.  The quadratic identity above then yields
$\Phi(z,p)=-ar$.
Theorem~\ref{thm:e2-main} therefore implies that $F_{p,2}$ is reducible.
\end{proof}

By Proposition~\ref{prop:q2-cardinality}, at most two associated
quadratic forms can occur.  Consequently, once $\mathcal Q_2(f;m)$ has
been constructed, the criterion above reduces the reducibility problem for
large primes to at most two fixed Pell-type equations.

The arithmetic nature of \eqref{eq:q2-associated-pell} depends on its
discriminant.  If $\Delta(\Phi)<0$, the equation is positive definite and
has only finitely many integral solutions.  If $\Delta(\Phi)$ is a square,
including the degenerate case $\Delta(\Phi)=0$, the equation reduces to an
elementary factorization or square condition.  The genuinely Pellian case
is $\Delta(\Phi)>0$ with $\Delta(\Phi)$ nonsquare.  It may have infinitely
many integral solutions, but determining which of their second
coordinates are prime is a separate arithmetic problem; our criterion
decides reducibility for each prescribed prime $p$.

\begin{example}[Empty and maximal associated sets]\label{ex:q2-sets}
Let $m=2$, $n=4$, and $r=1$, so $M=1$, $N=2$, and $W=x$.

First take $f(x)=x^4+x+1$.
Any normalized realizing polynomial would have the form
$h=x^2+u$ with $u\in\Q$.  In a representation $f=h^2+bWh+cW^2$,
comparison at degree $3$ gives $b=0$, after which the coefficient of
$x$ on the right is $0$, a contradiction.  Hence
$\mathcal Q_2(f;2)=\varnothing$.  Here $H(f)=1$, $R=3$, and
$P_2(1,f)=34$, so Theorem~\ref{thm:e2-main} gives the concrete family
$x^2+p^2(x^4+x+1)$ of irreducible polynomials over $\Q$ for every prime
$p>34$.

Now take $f(x)=x^4+1$.  The two normalized polynomials
$h_\pm=x^2\pm1$ give
$\mathcal Q_2(f;2)=\{X^2-2Y^2,\ X^2+2Y^2\}$, so the upper bound in
Proposition~\ref{prop:q2-cardinality} is attained.
Again $P_2(1,f)=34$.  The prime $p=5741$ and the integer $z=8119$
satisfy $z^2-2p^2=-1$.
Consequently
\[
\begin{aligned}
x^2+5741^2(x^4+1)
={}&\bigl(5741(x^2+1)+8119x\bigr)\\
&\mathrel{}\cdot
\bigl(5741(x^2+1)-8119x\bigr),
\end{aligned}
\]
which illustrates both directions of Theorem~\ref{thm:e2-main} above
the explicit threshold.
\end{example}

\section{Effective reducibility criteria via binary cubic forms for
\texorpdfstring{$F_{p,3}$}{Fp3}}
\label{sec:e3}

We now specialize to the case $e=3$, so that
$F_{p,3}(x)=rx^m+p^3f(x)$.
By Corollary~\ref{cor:delta-one}, if at least one of $m$ and $n$ is not
divisible by $3$, then $F_{p,3}$ is irreducible over $\Q$ for every
prime $p>P_*$.  Hence the only case requiring further analysis is
$3\mid m$ and $3\mid n$.  Accordingly, throughout this section we write
$m=3M$ and $n=3N$, and put $W:=x^M$ and $A:=\lc(f)$.

The proof of the criterion below will show that reducibility forces $f$
to admit a cubic representation in a monic degree-$N$ polynomial and
$W=x^M$.  We therefore introduce the corresponding binary cubic forms.

For $b,c,d\in\Q$, put
\begin{equation}\label{eq:Phi3}
\Phi_{b,c,d}(X,Y)
:=
X^3-bX^2Y+AcXY^2-A^2dY^3.
\end{equation}
Define $\mathcal Q_3(f;m)$ to be the set of all such forms for which
there exists a monic polynomial $h\in\Q[x]$ satisfying
\begin{equation}\label{eq:q3-hnorm}
\deg h=N,
\qquad
[x^M]h=0,
\end{equation}
The same polynomial must also satisfy
\begin{equation}\label{eq:q3-polyrep}
\Phi_{b,c,d}(Ah,-W)=A^2f.
\end{equation}
Equivalently, $f=Ah^3+bWh^2+cW^2h+dW^3$.
We call the elements of $\mathcal Q_3(f;m)$ the cubic forms associated
with $f$ and $m$, and say that such a polynomial $h$ \emph{realizes}
the corresponding form.  By construction, $\mathcal Q_3(f;m)$ depends
only on $f$ and $m$, and is independent of $r$ and $p$.

Put
\begin{equation}\label{eq:C3}
C_3:=(1+R)^N
\end{equation}
and define
\begin{equation}\label{eq:P3}
P_3(r,f)
:=\max\left\{
3,\ P_*,\
H(f)^4\bigl((2N+4)C_3^2+(N+2)C_3^3\bigr)
\right\}.
\end{equation}

The cubic necessity argument requires one additional lifting step that
has no analogue in the quadratic case.  We isolate the two elementary
tools needed for this step: cancellation modulo a prime power and a
bounded lifting from a modular span to an exact rational span.

\begin{lemma}[Cancellation modulo a prime power]\label{lem:padic-cancel}
Let $P,Q\in\Z[x]$, let $p$ be prime, and let $k\ge1$.  If
$P\not\equiv0\pmod p$ and $PQ\equiv0\pmod{p^k}$, then
$Q\equiv0\pmod{p^k}$.
\end{lemma}

\begin{proof}
We argue by induction on $k$.  For $k=1$, the assertion follows because
$\mathbb F_p[x]$ is an integral domain.  Let $k>1$.  Reducing modulo $p$
gives $\overline P\,\overline Q=0$ in $\mathbb F_p[x]$.  Since
$\overline P\ne0$, one has
$\overline Q=0$, so $Q=pQ_1$ for some $Q_1\in\Z[x]$.  Then
$pPQ_1\equiv0\pmod{p^k}$, and hence
$PQ_1\equiv0\pmod{p^{k-1}}$.
The induction hypothesis gives $Q_1\equiv0\pmod{p^{k-1}}$, and therefore
$Q\equiv0\pmod{p^k}$.
\end{proof}

\begin{lemma}[Bounded lifting in the cubic case]\label{lem:cubic-lifting}
Let $p$ be a prime.  Let $0\le M<N$, put $W=x^M$, and let
$H_0,K\in\Z[x]$ satisfy $\deg H_0=N$, $\deg K=2N$,
$\alpha:=\lc(H_0)\ne0$, $[x^M]H_0=0$, and $p\nmid\alpha$.
Suppose that, modulo $p^3$,
\begin{equation}\label{eq:cubic-modspan}
K\in\operatorname{span}_{\Z/p^3\Z}\{W^2,WH_0,H_0^2\}.
\end{equation}
Here and below, a polynomial congruence modulo $p^3$ means coefficientwise
congruence in $(\Z/p^3\Z)[x]$.
Let $\mathcal H,\mathcal C\ge1$ and assume $|\alpha|\le\mathcal H$,
$H(H_0)\le\mathcal H\mathcal C$,
$H(K)\le p^2\mathcal H\mathcal C^2$, and
$|\lc(K)|\le p^2\mathcal H$.
If
\begin{equation}\label{eq:cubic-lift-threshold}
p>
\mathcal H^4\bigl((2N+4)\mathcal C^2+(N+2)\mathcal C^3\bigr),
\end{equation}
then $K\in\operatorname{span}_{\Q}\{W^2,WH_0,H_0^2\}$.
\end{lemma}

\begin{proof}
Write $H_0^2=\sum_jq_jx^j$ and $K=\sum_jk_jx^j$.
Because $N>M$ and $[x^M]H_0=0$, the three basis elements have
distinct leading degrees $2M<N+M<2N$.  Define
$\mathsf G:=\alpha k_{2N}$,
$\mathsf B:=\alpha^2k_{N+M}-k_{2N}q_{N+M}$, and
$\mathsf A:=\alpha^3k_{2M}-\alpha k_{2N}q_{2M}$, and put
$E:=\alpha^3K-\mathsf A W^2-\mathsf B WH_0-\mathsf G H_0^2$.
By \eqref{eq:cubic-modspan}, there exist residue classes
$a_0,a_1,a_2\in\Z/p^3\Z$ such that
$K\equiv a_0W^2+a_1WH_0+a_2H_0^2\pmod{p^3}$.
Let $A_0=\alpha^3a_0$, $B_0=\alpha^3a_1$, and
$G_0=\alpha^3a_2$ in $\Z/p^3\Z$.  Then
\begin{equation}\label{eq:cubic-triangular-representation}
\alpha^3K\equiv A_0W^2+B_0WH_0+G_0H_0^2\pmod{p^3}.
\end{equation}
Comparison at degree $2N$ in
\eqref{eq:cubic-triangular-representation} gives
$\alpha^3k_{2N}\equiv G_0\alpha^2\pmod{p^3}$.
Since $p\nmid\alpha$, cancellation of $\alpha^2$ yields
$G_0\equiv\alpha k_{2N}=\mathsf G\pmod{p^3}$.
At degree $N+M$, only $WH_0$ and $H_0^2$ contribute, and hence
$\alpha^3k_{N+M}\equiv B_0\alpha+G_0q_{N+M}\pmod{p^3}$.
Substituting the preceding congruence and cancelling $\alpha$ gives
\[
B_0\equiv
\alpha^2k_{N+M}-k_{2N}q_{N+M}
=\mathsf B
\pmod{p^3}.
\]
Finally, $[x^M]H_0=0$ implies
$[x^{2M}](WH_0)=0$.  Comparison at degree $2M$ therefore gives
$\alpha^3k_{2M}\equiv A_0+G_0q_{2M}\pmod{p^3}$,
and consequently
\[
A_0\equiv
\alpha^3k_{2M}-\alpha k_{2N}q_{2M}
=\mathsf A
\pmod{p^3}.
\]
Thus the three coefficients in
\eqref{eq:cubic-triangular-representation} are congruent to
$\mathsf A,\mathsf B,\mathsf G$, respectively, and
$E\equiv0\pmod{p^3}$.

Furthermore, $H(H_0^2)\le(N+1)\mathcal H^2\mathcal C^2$, because every
coefficient of $H_0^2$ is a sum of at most $N+1$
products of coefficients of $H_0$.  The assumed bounds now give
$|\mathsf G|\le p^2\mathcal H^2$,
\[
\begin{aligned}
|\mathsf B|
&\le |\alpha|^2H(K)
   +|\lc(K)|H(H_0^2)\\
&\le p^2(N+2)\mathcal H^3\mathcal C^2,
\end{aligned}
\]
\[
\begin{aligned}
|\mathsf A|
&\le |\alpha|^3H(K)
   +|\alpha|\,|\lc(K)|H(H_0^2)\\
&\le p^2(N+2)\mathcal H^4\mathcal C^2.
\end{aligned}
\]
Consequently, the four summands in $E$ satisfy
\[
\begin{aligned}
H(\alpha^3K)&\le p^2\mathcal H^4\mathcal C^2,\qquad
H(\mathsf A W^2)\le p^2(N+2)\mathcal H^4\mathcal C^2,\\
H(\mathsf B WH_0)&\le p^2(N+2)\mathcal H^4\mathcal C^3,\qquad
H(\mathsf G H_0^2)\le p^2(N+1)\mathcal H^4\mathcal C^2.
\end{aligned}
\]
Adding these estimates yields
\[
H(E)
\le
p^2\mathcal H^4
\bigl((2N+4)\mathcal C^2+(N+2)\mathcal C^3\bigr)
<p^3.
\]
Every coefficient of $E$ is therefore both divisible by $p^3$ and smaller
than $p^3$ in absolute value.  Hence $E=0$, proving the assertion.
\end{proof}

\begin{theorem}[Cubic-form criterion]
\label{thm:e3-main}
Let $p>P_3(r,f)$ be prime.  Then $F_{p,3}$ is reducible over $\Q$ if and
only if there exists $\Phi\in\mathcal Q_3(f;m)$ such that
\begin{equation}\label{eq:q3-arithmetic}
\Phi(z,p)=A^2r
\end{equation}
for some $z\in\Z$.
\end{theorem}

\begin{proof}
Suppose first that $\Phi\in\mathcal Q_3(f;m)$ and that $z\in\Z$ satisfies
\eqref{eq:q3-arithmetic}.  Write $\Phi=\Phi_{b,c,d}$ and choose a
polynomial $h$ realizing $\Phi$.  Then
$f=Ah^3+bWh^2+cW^2h+dW^3$ and
$z^3-pbz^2+p^2Acz-p^3A^2d=A^2r$.
A direct expansion gives
\[
\begin{aligned}
&(pAh+zW)
\Bigl(
 p^2A^2h^2+pA(pb-z)Wh
 +(z^2-pbz+p^2Ac)W^2
\Bigr)\\
&\qquad
=p^3A^2\bigl(Ah^3+bWh^2+cW^2h+dW^3\bigr)+A^2rW^3\\
&\qquad
=p^3A^2f+A^2rx^m
=A^2F_{p,3}.
\end{aligned}
\]
Since $A\ne0$ and $N>M$, the two displayed factors have degrees $N$ and
$2N$, respectively.  Dividing one factor by the nonzero scalar $A^2$
therefore gives a nontrivial factorization of $F_{p,3}$ over $\Q$.

Conversely, suppose that $F_{p,3}$ is reducible.

\smallskip
\noindent
\emph{Step 1: the degree-$N$ factor and its normalization.}
Since $p>P_3(r,f)\ge P_*$, Lemma~\ref{lem:integral} gives a
nontrivial factorization in $\Z[x]$.  Here
$\delta=\gcd(3,m,n)=3$.  Theorem~\ref{thm:factor-structure} shows that a
proper factor has parameter $t=1$ or $2$; after interchanging the two
factors if necessary, we may therefore write
\begin{equation}\label{eq:q3-GK}
F_{p,3}=GK,
\qquad
G,K\in\Z[x],
\qquad
\deg G=N,
\qquad
\deg K=2N.
\end{equation}
Applying Theorem~\ref{thm:factor-structure} to $G$ and $K$ also gives
$v_p(\lc(G))=1$ and $v_p(\lc(K))=2$.  Moreover,
$G\equiv uW\pmod p$ and $K\equiv vW^2\pmod p$ for some integers $u,v$
with $p\nmid uv$, while $H(G)\le pH(f)C_3$ and
$H(K)\le p^2H(f)C_3^2$.  Write $G=pA_0+uW$, and set
$c_0:=[x^M]A_0$, $H_0:=A_0-c_0W$, and $s:=u+pc_0$.
Then
\begin{equation}\label{eq:q3-Gnormalized}
G=pH_0+sW,
\qquad
[x^M]H_0=0.
\end{equation}
Let $\alpha:=\lc(H_0)$ and write $\lc(K)=p^2\beta$.
Since $N>M$, one has $\lc(G)=p\alpha$, so the valuation information above
gives $p\nmid\alpha\beta$.  Comparing leading coefficients in
\eqref{eq:q3-GK} yields
\begin{equation}\label{eq:q3-alphabeta}
\alpha\beta=A,
\qquad
1\le|\alpha|,|\beta|\le H(f).
\end{equation}
Put $h:=H_0/\alpha$.  Then $h$ is monic of degree $N$ and satisfies
\eqref{eq:q3-hnorm}.

\smallskip
\noindent
\emph{Step 2: lifting the complementary factor.}
Since \eqref{eq:q3-Gnormalized} gives
$[x^j]H_0=[x^j]G/p$ for $j\ne M$, while $[x^M]H_0=0$, we obtain
\begin{equation}\label{eq:q3-heightbounds}
H(H_0)\le H(f)C_3,
\qquad
H(K)\le p^2H(f)C_3^2.
\end{equation}
Moreover, $|\lc(K)|=p^2|\beta|\le p^2H(f)$.
Reducing \eqref{eq:q3-GK} modulo $p^3$ gives
$(pH_0+sW)K\equiv rW^3\pmod{p^3}$.
Since $s\equiv u\not\equiv0\pmod p$, the class of $s$ is invertible
modulo $p^3$.  A direct multiplication gives
\[
(pH_0+sW)
\bigl(s^{-1}W^2-ps^{-2}WH_0+p^2s^{-3}H_0^2\bigr)
\equiv W^3\pmod{p^3},
\]
where the inverses of $s$ are taken in $\Z/p^3\Z$.  Hence
\[
(pH_0+sW)
\left[
K-r\bigl(s^{-1}W^2-ps^{-2}WH_0+p^2s^{-3}H_0^2\bigr)
\right]
\equiv0\pmod{p^3}.
\]
Since $pH_0+sW\equiv sW\not\equiv0\pmod p$,
Lemma~\ref{lem:padic-cancel} yields
\begin{equation}\label{eq:q3-Kmodp3}
K\equiv
r\bigl(s^{-1}W^2-ps^{-2}WH_0+p^2s^{-3}H_0^2\bigr)
\pmod{p^3}.
\end{equation}
Thus the modular span hypothesis of Lemma~\ref{lem:cubic-lifting} holds.
Apply that lemma with $\mathcal H=H(f)$ and $\mathcal C=C_3$.
By \eqref{eq:P3}, its size condition is satisfied, and therefore
\begin{equation}\label{eq:q3-Kexact}
K=q_0W^2+q_1Wh+q_2h^2
\end{equation}
for some $q_0,q_1,q_2\in\Q$.  Since $h$ is monic and $N>M$, the term
$q_2h^2$ is the unique contribution of degree $2N$, so
$q_2=\lc(K)=p^2\beta$.

\smallskip
\noindent
\emph{Step 3: reconstruction of the binary cubic form.}
Since $H_0=\alpha h$, \eqref{eq:q3-Gnormalized} becomes
$G=p\alpha h+sW$.
Substituting this and \eqref{eq:q3-Kexact} into $GK=F_{p,3}$ gives
\[
\begin{aligned}
GK={}&p^3Ah^3
 +(p\alpha q_1+sp^2\beta)Wh^2\\
&+(p\alpha q_0+sq_1)W^2h
 +sq_0W^3.
\end{aligned}
\]
The four basis elements $h^3$, $Wh^2$, $W^2h$, and $W^3$ have pairwise
distinct degrees.  Hence, defining
\begin{equation}\label{eq:q3-bcd}
b:=\frac{p\alpha q_1+sp^2\beta}{p^3},
\qquad
c:=\frac{p\alpha q_0+sq_1}{p^3},
\qquad
d:=\frac{sq_0-r}{p^3},
\end{equation}
comparison with
$F_{p,3}=rW^3+p^3f$ yields
$f=Ah^3+bWh^2+cW^2h+dW^3$.  Thus the form
$\Phi(X,Y):=X^3-bX^2Y+AcXY^2-A^2dY^3$ belongs to
$\mathcal Q_3(f;m)$.

It remains to prove the arithmetic representation.  Multiplying the
degree-$N$ factor by $\beta$, put $\beta G=pAh+zW$, where
$z:=\beta s\in\Z$.
A direct expansion, using the representation above, gives
\begin{equation}\label{eq:q3-divisibility-identity}
\begin{aligned}
&A^2F_{p,3}
-(pAh+zW)
\Bigl(
 p^2A^2h^2+pA(pb-z)Wh
 +(z^2-pbz+p^2Ac)W^2
\Bigr)\\
&\qquad
=W^3\bigl(A^2r-\Phi(z,p)\bigr).
\end{aligned}
\end{equation}
Since $A=\alpha\beta$, the polynomial
$pAh+zW=\beta G$ divides $A^2F_{p,3}=A^2GK$ and therefore divides the
left-hand side of \eqref{eq:q3-divisibility-identity}.  Hence it also
divides the right-hand side.  If $M=0$, then $W=1$, so a nonconstant polynomial cannot
divide a nonzero constant.  If $M>0$, the representation gives
$f(0)=Ah(0)^3\ne0$, so $h(0)\ne0$ and therefore
$\gcd(pAh+zW,W)=1$.
Again a nonconstant polynomial coprime to $W$ cannot divide a nonzero
constant multiple of $W^3$.  Thus $\Phi(z,p)=A^2r$, which proves
\eqref{eq:q3-arithmetic}.

\end{proof}

\begin{remark}[Monic case]\label{rem:q3-monic}
If $f$ is monic and $F_{p,3}$ is reducible, the necessity proof above
produces the associated form with $h\in\Z[x]$ and $b,c,d\in\Z$.
Indeed, $A=1$ and \eqref{eq:q3-alphabeta} give
$\alpha,\beta\in\{\pm1\}$, so $h\in\Z[x]$.  Since
$2N+M>N+2M>3M$,
comparison successively at degrees $2N+M$, $N+2M$, and $3M$ in
$f=h^3+bWh^2+cW^2h+dW^3$
shows that $b,c,d\in\Z$.
\end{remark}

Theorem~\ref{thm:e3-main} reduces the reducibility problem for large primes to
the cubic forms associated with $f$ and $m$.  In contrast with the
quadratic case, there is never more than one such form.  Notice also that
if $\mathcal Q_3(f;m)=\varnothing$, then Theorem~\ref{thm:e3-main}
immediately gives the irreducibility of $F_{p,3}$ for every prime
$p>P_3(r,f)$.

\begin{proposition}[Cardinality of the associated cubic forms]
\label{prop:q3-cardinality}
The set $\mathcal Q_3(f;m)$ satisfies
$|\mathcal Q_3(f;m)|\le1$.
More precisely, whenever $\mathcal Q_3(f;m)$ is nonempty, both the associated form and the polynomial realizing it are unique.
\end{proposition}

\begin{proof}
Suppose that the two representations
\[
\begin{aligned}
f&=Ah^3+bWh^2+cW^2h+dW^3,\\
f&=A\widetilde h^{\,3}
 +\widetilde bW\widetilde h^{\,2}
 +\widetilde cW^2\widetilde h
 +\widetilde dW^3
\end{aligned}
\]
satisfy the defining conditions.

If $M=0$, then $h(0)=\widetilde h(0)=0$.  If $h\ne\widetilde h$, put
$\ell:=\deg(h-\widetilde h)$.
Since both polynomials are monic of degree $N$ and have zero constant
term, one has
$1\le\ell<N$.  Subtracting the two representations gives
\[
\begin{aligned}
A(h-\widetilde h)(h^2+h\widetilde h+\widetilde h^2)
={}&\widetilde b\widetilde h^2-bh^2\\
&+\widetilde c\widetilde h-ch+(\widetilde d-d).
\end{aligned}
\]
The left-hand side has degree $\ell+2N>2N$, whereas the right-hand side
has degree at most $2N$, a contradiction.  Hence $h=\widetilde h$.
The distinct degrees $2N$, $N$, and $0$ then give
$b=\widetilde b$, $c=\widetilde c$, and $d=\widetilde d$.

Assume now that $M>0$.  Since $f(0)\ne0$, evaluation at $x=0$ gives
$h(0)=\widetilde h(0)\ne0$.
Indeed, $Ah(0)^3=A\widetilde h(0)^3=f(0)$, and the map
$u\mapsto u^3$ is injective on $\Q$.
Subtracting the two representations yields
\[
\begin{aligned}
A(h-\widetilde h)(h^2+h\widetilde h+\widetilde h^2)
={}&W(\widetilde b\widetilde h^2-bh^2)\\
&+W^2(\widetilde c\widetilde h-ch)
 +(\widetilde d-d)W^3.
\end{aligned}
\]
The second factor on the left has nonzero constant term
$3h(0)^2$ and is therefore coprime to $x$.  Since the right-hand side is
divisible by $W=x^M$, it follows that $W\mid h-\widetilde h$.  If
$h\ne\widetilde h$ and $\ell=\deg(h-\widetilde h)$, then
$\ell\ge M$.  The left-hand side has degree $\ell+2N$, while the
right-hand side has degree at most $M+2N$.  Thus $\ell\le M$, and hence
$\ell=M$.  But a nonzero polynomial divisible by $x^M$ and of degree $M$
has a nonzero $x^M$ coefficient, contradicting
$[x^M]h=[x^M]\widetilde h=0$.  Therefore $h=\widetilde h$.  Since
$2N+M>N+2M>3M$,
comparison of the three remaining leading degrees gives
$b=\widetilde b$, $c=\widetilde c$, and $d=\widetilde d$.
Thus the realizing polynomial, and hence the associated binary cubic
form, is unique.
\end{proof}

\medskip
\noindent\textbf{Effective construction of $\mathcal Q_3(f;m)$.}
The associated set is also effectively computable from $f$ and $m$.
Writing $h(x)=x^N+\sum_{j=0}^{N-1}u_jx^j$ with $u_M=0$,
comparison at degree $2N+j$, for $j=N-1,\ldots,M+1$, has the form
$f_{2N+j}/A=3u_j+R_j$,
where $R_j$ depends only on $u_{j+1},\ldots,u_{N-1}$.  Hence
$u_{N-1},\ldots,u_{M+1}$ are recovered successively.  If $M=0$, the
normalization gives $u_0=0$.  If $M>0$, the constant term gives
$u_0^3=f(0)/A$, which has at most one rational solution.  Since
$f(0)\ne0$, any such
solution satisfies $u_0\ne0$.  Once $u_0$ is fixed, comparison at degree
$j$, for $1\le j\le M-1$, has the form
$f_j/A=3u_0^2u_j+S_j$,
where $S_j$ depends only on $u_0,\ldots,u_{j-1}$; hence
$u_1,\ldots,u_{M-1}$ are recovered successively.  Once $h$ is fixed, the
coefficients $b,c,d$ are uniquely determined by comparison at degrees
$2N+M$, $N+2M$, and $3M$, followed by one final identity check.  Thus
$\mathcal Q_3(f;m)$ is obtained from a single candidate polynomial,
independently of $p$ and $r$.

Finally, the unique associated cubic form, when it exists, leads to a
fixed binary cubic Diophantine equation.  After clearing denominators, this
becomes a Thue equation whenever the resulting integral binary cubic form is
irreducible.  Let $\Phi\in\mathcal Q_3(f;m)$, and choose a positive
integer $q$ such that $\Psi(X,Y):=q\Phi(X,Y)\in\Z[X,Y]$.
Multiplication by a nonzero rational scalar does not affect
irreducibility over $\Q$.  Thus the irreducibility hypothesis below is
independent of the choice of the denominator-clearing integer $q$.

\begin{corollary}[Finiteness via Thue's theorem]\label{cor:q3-thue}
Assume that $\mathcal Q_3(f;m)$ is nonempty and that the integral binary
cubic form $\Psi$ defined above is irreducible over $\Q$.  Then only
finitely many primes $p>P_3(r,f)$ make $F_{p,3}$ reducible over $\Q$.
\end{corollary}

\begin{proof}
By Theorem~\ref{thm:e3-main}, reducibility gives an integer $z$ such that
$\Phi(z,p)=A^2r$, and hence $\Psi(z,p)=qA^2r$.
The right-hand side is a fixed nonzero integer.  Since $\Psi$ is an
irreducible binary cubic form, Thue's theorem~\cite{Thue1909} gives only
finitely many integer solutions $(z,p)$.  Therefore only finitely many
primes $p$ can occur.
\end{proof}

The conclusion of Corollary~\ref{cor:q3-thue} is qualitative: it does not
assert that the largest reducible prime is bounded by the explicit
threshold $P_3(r,f)$.  Effective upper bounds for the solutions of the
resulting Thue equation require additional Diophantine estimates that are
not part of the present paper.

By Proposition~\ref{prop:q3-cardinality}, at most one associated cubic
form can occur.  Thus, once $\mathcal Q_3(f;m)$ has been constructed, the
reducibility problem for large primes is governed by at most one fixed
binary cubic equation.

\begin{example}[Reducible and irreducible associated cubic forms]
\label{ex:q3-forms}
Let $m=3$, $n=6$, $r=1$, $M=1$, $N=2$, $W=x$, and
$h=x^2+1$.

First take $f(x)=h(x)^3=(x^2+1)^3$.
Then the unique associated form is $\Phi(X,Y)=X^3$, and
$\Phi(1,p)=1$ for every prime $p$.  In fact, for every $p$ one has the
identity
\[
\begin{aligned}
x^3+p^3(x^2+1)^3
={}&\bigl(x+p(x^2+1)\bigr)\\
&\mathrel{}\cdot
\bigl(x^2-px(x^2+1)+p^2(x^2+1)^2\bigr).
\end{aligned}
\]
Thus a reducible associated cubic form can lead to infinitely many
reducible primes; the irreducibility hypothesis in
Corollary~\ref{cor:q3-thue} is essential.  The identity also illustrates
that the explicit threshold in Theorem~\ref{thm:e3-main} is a sufficient
uniform threshold and is not expected to be sharp in every family.

For a contrasting example, take
$f(x)=h^3-W^2h+W^3=x^6+2x^4+x^3+2x^2+1$.  The unique associated form is
$\Phi(X,Y)=X^3-XY^2-Y^3$.
It is irreducible over $\Q$, since the cubic polynomial
$\Phi(X,1)=X^3-X-1$ has no rational root.  Here $H(f)=2$, $R=4$,
$C_3=25$, and $P_3(1,f)=1{,}080{,}000$.
Corollary~\ref{cor:q3-thue} therefore shows that only finitely many
primes $p>1{,}080{,}000$ can make
$x^3+p^3(x^6+2x^4+x^3+2x^2+1)$ reducible over $\Q$.
\end{example}

\section{Concluding remarks on higher exponents}
\label{sec:conclusion}

The quadratic and cubic arguments are both driven by the proper-factor
structure established in Theorem~\ref{thm:factor-structure}.  For
arbitrary $e\ge1$ and sufficiently large primes $p$, every nonconstant
proper integral factor $G$ is governed by
$\delta=\gcd(e,m,n)$ and a parameter $t$ satisfying
$1\le t\le\delta-1$.  More precisely,
its degree is $tn/\delta$, its endpoint valuations are explicit, its
reduction modulo $p$ is the single monomial
$bx^{tm/\delta}$, and its height is bounded by an explicit multiple of
$p^{te/\delta}$.  These four pieces of information provide exactly the
input needed for the reconstruction arguments in
Sections~\ref{sec:e2} and~\ref{sec:e3}.

For $e=2$, reducibility forces $\delta=2$, so every proper factor has
$t=1$ and hence degree $n/2$.  For $e=3$, reducibility forces
$\delta=3$, and the two possible parameters are $1$ and $2$; in any
nontrivial two-factor decomposition, after interchanging the factors if
necessary, a factor with $t=1$ and degree $n/3$ is therefore available.
The endpoint valuations and the monomial reductions supplied by
Theorem~\ref{thm:factor-structure} normalize these factors, while its
height estimate gives the coefficient control required to pass from
congruence information to exact polynomial identities.

The two completed cases consequently have parallel final forms.  For
$e=2$, the polynomial data determine a set $\mathcal Q_2(f;m)$ of at
most two associated binary quadratic forms, and
Theorem~\ref{thm:e2-main} characterizes reducibility by an arithmetic
representation of $-ar$ by one of these forms with second coordinate
$p$.  Equivalently, Corollary~\ref{cor:q2-pell} reduces the problem to
at most two fixed Pell-type equations.  For $e=3$, the corresponding
set $\mathcal Q_3(f;m)$ contains at most one associated binary cubic
form.  When this set is nonempty, Theorem~\ref{thm:e3-main}
characterizes reducibility by a single binary cubic equation, while
Corollary~\ref{cor:q3-thue} gives finiteness of the reducible primes
when an integral multiple of the associated cubic form is irreducible.
Thus, in both cases, the reducibility problem for sufficiently large
primes separates into a polynomial preprocessing step depending only on
$f$ and $m$, followed by a fixed Diophantine equation in which $p$
appears as one coordinate.

Two independent difficulties arise for higher exponents.  The first is
structural.  When $\delta\ge4$, Theorem~\ref{thm:factor-structure} does
not force the existence of a factor with $t=1$.  For example, when
$\delta=4$, the theorem permits a two-factor degree configuration with
$t=2$ for both factors, namely $\deg G=\deg K=n/2$.
Thus a factor of the smallest possible degree $n/\delta$ need not be
available, and the reconstruction used for $e=2$ and $e=3$ may have no
canonical starting factor.

The second difficulty concerns $p$-adic depth.  Write $e=\delta E$.
A factor corresponding to $t=1$ has endpoint valuation $E$ by
Theorem~\ref{thm:factor-structure}.  In the quadratic and cubic cases
considered here, $E=1$, and the reduction modulo $p$ gives precisely the
first-order shape needed to start the reconstruction.  When $E>1$,
however, the same reduction controls only the first $p$-adic layer,
whereas an analogous reconstruction is expected to require information
modulo higher powers of $p$.  Such higher-depth information is not
determined by the Newton polygon alone and must be recovered by
additional congruence and height arguments.

These observations suggest two natural directions.  Put
$W_\delta:=x^{m/\delta}$.
For a proper factor corresponding to a parameter $t$, one may ask
whether, under suitable effective hypotheses, it must lie in the
$\Q$-linear span
\[
\operatorname{span}_{\Q}
\{W_\delta^t,W_\delta^{t-1}h,\ldots,h^t\}
\]
for a monic polynomial $h$ of degree $n/\delta$ satisfying an
appropriate normalization.  Independently, when $e/\delta>1$, one must
determine what additional congruence information and effective prime
bounds permit the necessary lifting through higher powers of $p$.  The
exponent $4$ is the first natural testing ground for both issues: when
$\delta=4$, a two-factor decomposition with $t=2$ on both sides is not
excluded, whereas $\delta=2$ gives $E=2$ and hence the first case in
which higher $p$-adic depth is unavoidable.

\section*{Acknowledgements}
Hongjian Li was supported by the Project of Guangdong University of
Foreign Studies (Grant No.~2024RC063).

\end{document}